\documentclass{amsart}
\usepackage{amsfonts}
\usepackage{amsmath,amsthm}
\usepackage{amsfonts,amssymb}
\usepackage{mathrsfs}
\usepackage{enumerate}

\newtheorem{thm}{Theorem}[section]

\newtheorem{lem}[thm]{Lemma}

\newtheorem{defn}[thm]{Definition}

\theoremstyle{remark}
\newtheorem{rem}[thm]{Remark}

\numberwithin{equation}{section}

\newcommand{\al}{\alpha}

\def\vz{\varepsilon}

\def\lz{\lambda}
\def\Lz{\Lambda}

\def\az{\alpha}
\def\gz{\gamma}

\def\sz{\sigma}

\def\({\Bigl(}
\def \){ \Bigr)}

 \def\bb{{\Bbb B^d}}
 
 \def\rb{{\mathbf r}}

 \def\RR{{\mathbb R}}

\def\sz{\sigma}

\def\ss{{\Bbb S}^{d}}

\def\lb{\langle}
\def\rb{\rangle}

\def\bi{\bibitem}

\def\Lz{\Lambda}

\begin{document}
\def\RR{\mathbb{R}}
\def\Exp{\text{Exp}}
\def\FF{\mathcal{F}_\al}

\title[] {Approximation numbers of  Sobolev and Gevrey type embeddings on the sphere and on the ball --  Preasymptotics,  asymptotics, and tractability}

\author[]{Jia Chen, Heping Wang} \address{School of Mathematical Sciences, Capital Normal
University, Beijing 100048,
 China.}
\email{ jiachencd@163.com;\ \  \ wanghp@cnu.edu.cn.}

%\date{\today}
\keywords{Approximation number;  Sobolev spaces; Gevrey type
spaces, preasymptotics and asymptotics; tractability}

\begin{abstract} In this paper, we  investigate optimal linear approximations ($n$-approximation
numbers )    of the embeddings from the  Sobolev spaces  $H^r\
(r>0)$ for various equivalent norms and the Gevrey type spaces
$G^{\az,\beta}\ (\az,\beta>0)$ on the sphere $\ss$ and on the ball
$\bb$, where the approximation error is measured in the
$L_2$-norm.  We obtain preasymptotics, asymptotics, and strong
equivalences of the above approximation numbers as  a function in
$n$ and the dimension $d$. We  emphasis that  all equivalence
constants in the above preasymptotics and asymptotics are
independent of  the dimension $d$ and $n$. As a consequence we
obtain that  for the absolute error criterion the approximation
problems $I_d: H^{r}\to L_2$
 are weakly tractable if and only if $r>1$, not uniformly
weakly tractable,  and do not suffer from the curse of
dimensionality. We also prove that  for any $\az,\beta>0$, the
approximation problems $I_d: G^{\az,\beta}\to L_2$ are  uniformly
weakly tractable, not polynomially tractable, and
quasi-polynomially tractable if and only if $\az\ge 1$.

\end{abstract}

\maketitle
\input amssym.def

\section{Introduction}

This paper is  devoted to investigating the behavior of the
approximation numbers of embeddings of  Sobolev spaces  and Gevrey
type spaces on the sphere $\ss$ and on the ball $\bb$ into $L_2$.
The approximation numbers of a bounded linear operator
$T:X\rightarrow Y$ between two Banach spaces are defined as
\begin{align*}\label{1.1}
a_n(T:X\rightarrow Y):&=\inf_{rank A< n}\sup_{\|x|X\|\leq 1}\|Tx-Ax|Y\|\notag\\
&=\inf_{rank A< n}\|T-A:X\rightarrow Y\|,\ \ n\in \Bbb N_+,
\end{align*}
where $\Bbb N_+=\{1,2,3,\dots\},\ \Bbb N=\{0,1,2,3,\dots\}$. They
describe the best approximation of $T$ by finite rank operators.
If $X$ and $Y$ are Hilbert spaces and $T$ is compact, then $a_n(T
)$ is the $n$th singular number of $T$. Also $a_n(T)$ is the $n$th
minimal worst-case error  with respect to arbitrary algorithms and
general information in the Hilbert setting.

 On the torus $\Bbb
T^d$, there are many results concerning  asymptotics of the
approximation numbers of smooth function spaces, see the
monographs \cite{Te} by Temlyakov  and the references therein.
However, the obtained asymptotics  often hide dependencies on the
 dimension $d$ in the constants, and can only be seen after
``waiting exponentially long ($n\ge 2^d)$'' if $d$ is large. In
order to overcome this deficiency, K\"uhn and other authors
obtained preasymptotics and asymptotics of the approximation
numbers of the classical isotropic Sobolev spaces, Sobolev spaces
of dominating mixed smoothness, periodic Gevrey spaces, and
anisotropic  Sobolev spaces (see \cite{KSU1, KSU2, KMU, CW}). Note
that in these preasymptotics and asymptotics, the equivalence
constants are independent of the dimension $d$ and $n$.

On the sphere $\ss$ and on the ball $\bb$, the following two-sided
estimates can be found in \cite{Ka} and \cite{WH} in a slightly
more general setting:
\begin{equation}\label{1.2}C_1(r,d)n^{-r/d}\le a_n(I_d:
H^r(\ss)\to L_2(\ss))\le C_2(r,d)n^{-r/d},\ n\in\Bbb N_+,
\end{equation}
and
\begin{equation}\label{1.3}C_3(r,d)n^{-r/d}\le a_n(I_d: H^r(\bb)\to L_2(\bb))\le C_4(r,d)n^{-r/d}, \ n\in\Bbb N_+,\end{equation}
 where $I_d$ is the
identity (embedding) operator, $H^r(\ss),\ H^r(\bb)$ are the
Sobolev spaces on the sphere $\ss$  and on the ball $\bb$, the
constants $C_i(r,d), i=1,2,3,4$,  only depending on the smoothness
index $r$ and the dimension
 $d$,
were not explicitly determined.

In the paper  we discuss preasymptotics and asymptotics of the
approximation numbers of the embeddings $I_d$ of the Sobolev
spaces $H^r\ (r>0)$ and the Gevrey type spaces $G^{\az,\beta}\
(\az,\beta>0)$  on the sphere $\ss$ and on the ball $\bb$ into
$L_2$. We remark that  the Gevrey type spaces have a long history
and have been used in numerous problems related to partial
differential equations.

Our main focus in this paper is to clarify, for arbitrary but
fixed $r>0$ and $\az,\beta>0$, the dependence of these
approximation numbers $a_n(I_d)$ on $d$. In fact, it is necessary
to fix the norms on the spaces $H^r$ on $\ss$ and on $\bb$
 in advance, since the constants $C_i(r,d),\ i=1,2,3,4
$ in \eqref{1.2} and \eqref{1.3} depend on the size of the
respective unit balls. Surprisingly, for a collection of quite
natural norms of $H^r$ (see Sections 2.1 and 6.1), and  for
sufficiently large $n$, say $n\ge 2^d$, it turns out that the
optimal constants decay polynomially in $d$, i.e.,
$$C_1(r,d)\asymp_r C_2(r,d)\asymp_r C_3(r,d)\asymp_r
 C_4(r,d)\asymp_r d^{-r},$$where
  $A\asymp B$ means that there exist two constants $c$ and $C$
which are called the equivalence constants such that $c A\le B\le
C A$, and $\asymp_{r}$ indicates that the equivalence constants
depend only on $r$. This means that on $\ss$ and on $\bb$, for
$n\ge 2^d$,
$$ a_n(I_d: H^r\to L_2)\asymp_r d^{-r}  n^{-r/d},$$where the equivalence
constants are independent of $d$ and $n$. We also show that  on
$\ss$ and on $\bb$, for $n\ge 2^d$,
$$ \ln( a_n(I_d: G^{\az,\beta}\to L_2))\asymp_\az -\beta  d^{\az}  n^{\az/d},$$
where the equivalence constants depend only on $\az$, but not on
$d$ and $n$.

Specially, we prove that the limits
\begin{equation*}\lim_{n\to\infty}n^{r/d}a_n(I_d:
H^{r}(\ss)\to
L_2(\ss))=\Big(\frac2{d\,!}\Big)^{r/d}\end{equation*} and
\begin{equation*}\lim_{n\to\infty}n^{r/d}a_n(I_d: H^{r}(\bb)\to
L_2(\bb))=\Big(\frac1{d\,!}\Big)^{r/d}\end{equation*} exist,
having the same value for various norms. We also prove that for
$0<\az<1$, $\beta>0$, $ \gamma=\Big(\frac2{d\,!}\Big)^{-\az/d}$,
$\tilde \gamma=\Big(\frac1{d\,!}\Big)^{-\az/d}$,
\begin{equation*}\lim_{n\to \infty }e^{\beta \gz n^{\az/d}}a_n(I_d: G^{\az, \beta}(\ss)\to L_2(\ss))=1.\end{equation*}
and
\begin{equation*}\lim_{n\to \infty }e^{\beta \tilde\gz n^{\az/d}}a_n(I_d: G^{\az, \beta}(\bb)\to L_2(\bb))=1.\end{equation*}

For small $n, 1\le n\le 2^d$, we also determine explicitly how
these approximation numbers $ a_n(I_d)$  behave preasymptotically.
We emphasize that the preasymptotic behavior of $ a_n(I_d)$ is
completely different from its asymptotic behavior.  For example,
we  show that
\begin{align*}
a_n(I_d: H^{r,*}(\ss)\to L_2(\ss))&\asymp_r a_n(I_d:
H^{r,*}(\bb)\to L_2(\bb))\\ & \asymp_r
\left\{\begin{matrix} 1,\ \ \ &n=1,\\
d^{-r/2}, & \ \  2\le n\leq d,\\
 d^{-r/2}\Big(\frac{\log(1+\frac{d}{\log n})}{\log n}\Big)^{r/2},&\ \  d\le n\le 2^d,
\end{matrix}\right.
\end{align*}
and \begin{align*} \ln \big(a_n(I_d: G^{\az,\beta}(\ss)\to
L_2(\ss))\big) &\asymp_\az \ln \big(a_n(I_d: G^{\az,\beta}(\bb)\to
L_2(\bb))\big) \\ &\asymp_\az -\beta
 \left\{\begin{matrix}
1, & \ \  1\le n\leq d,\\
 \Big(\frac{\log n}{\log(1+\frac{d}{\log n})}\Big)^{\az},&\ \  d\le n\le 2^d,
\end{matrix}\right.
\end{align*}
 where the equivalence constants depend only on $r$ or $\az$, but not
on $d$ and $n$. Here ``$*$'' stands for a specific (but natural)
norm in $H^r$ on $\ss$ and on $\bb$.

Finally we consider  tractability results for the approximation
problems of the Sobolev embeddings  and the Gevrey type embeddings
on $\ss$ and on $\bb$. Based on  the  asymptotic  and
preasymptotic behavior of $ a_n(I_d : H^{r}\to L_2)$ and
$a_n(I_d:G^{\az,\beta}\to L_2)$, we show that for the absolute
error criterion the approximation problems $I_d: H^{r}\to L_2$
 are weakly tractable if and only if $r>1$, not uniformly
weakly tractable,  and do not suffer from the curse of
dimensionality. We also prove that for any $\az,\beta>0$, the
approximation problems $I_d: G^{\az,\beta}\to L_2$ are   uniformly
weakly tractable, not polynomially tractable, and
quasi-polynomially tractable if and only if $\az\ge 1$ and  the
exponent of quasi-polynomial tractability is
$$t^{\rm qpol}=\sup_{m\in \Bbb N}\frac m {1+\beta m^\az},\ \
\az\ge1.$$

The paper is organized as follows. In Section 2.1 we give
definitions  of  the  Sobolev spaces with various equivalent norms
and the Gevrey type spaces on the sphere.
   Section 2.2 is
devoted to some basics on the approximation numbers on the sphere.
In Section 3, we study  strong equivalence of the approximation
numbers $ a_n(I_d : H^{r}(\ss)\to L_2(\ss))$ and
$a_n(I_d:G^{\az,\beta}(\ss)\to L_2(\ss))$. Section 4 contains
results concerning preasymptotics and asymptotics of the above
approximation numbers.  Section 5 transfers our approximation
results  into the tractability ones  of the respective
approximation problems.  In the final Section 6, we obtain the
corresponding results on $\bb$ such as strong equivalence,
preasymptotics and asymptotics of the approximation numbers $
a_n(I_d : H^{r}(\bb)\to L_2(\bb))$ and
$a_n(I_d:G^{\az,\beta}(\bb)\to L_2(\bb))$, and tractability of the
respective approximation problems.

\section{Preliminaries on the sphere}

\subsection{Sobolev spaces  and Gevrey type spaces on the  sphere}

\

Let $\ss=\{x\in\mathbb{R}^{d+1}:\, |x|=1\}$ ($d\ge 2$) be the unit
sphere of $\mathbb{R}^{d+1}$ (with $|x|$ denoting the Euclidean
norm in $\mathbb{R}^{d+1}$) endowed with the rotationally
invariant measure $\sz$ normalized by $\int_{\ss}d\sz=1$.  Denote
by $L_2(\ss)$ the collection of real measurable functions $f$ on
$\ss$ with finite norm
$$\|f|L_2(\ss)\|=\Big(\int_{\ss}|f(x)|^2\, d\sz(x)\Big)^\frac{1}{2}<+ \infty.$$

  We denote by $\mathcal{H}_\ell^d$ the space of
all spherical harmonics of degree $l$ on $\ss$. Denote by
$\Pi^{d+1}_m(\Bbb S^d)$  the set of spherical polynomials on
$\mathbb{S}^d$ of degree $\le m$, which is just the set of
polynomials of degree $\le m$ on $\mathbb{R}^{d+1}$ restricted to
$\mathbb{S}^d$. It is well known (see \cite[Corollaries 1.1.4 and
1.1.5]{DaiX}) that the dimension of $\mathcal{H}_\ell^d$ is
\begin{equation}\label{2.1} Z(d,\ell):={\rm
dim}\,\mathcal{H}_\ell^d=\left\{\begin{array}{cl} 1,\ \ \
   & {\rm if}\ \ \ell=0,\\
\frac{(2\ell +d-1)\,(\ell +d-2)!}{(d-1)!\ l!},\ \     & {\rm if}\
\ \ell=1,2,\dots,
\end{array}\right.\end{equation} and \begin{equation}\label{2.2}C(d,m):={\rm
dim}\,\Pi_m^{d+1}(\Bbb S^d)=\frac{(2m+d)(m+d-1)!}{m!\,d!},\
m\in\Bbb N.\end{equation}  Let
$$\{Y_{\ell,k}\equiv Y_{\ell,k}^{d}\ |\ k=1,\dots, Z(d,\ell)\}$$
be a fixed orthonormal basis for $\mathcal{H}_\ell^d$. Then
$$\{Y_{\ell,k}\ |\ k=1,\dots, Z(d,\ell),\ \ell=0,1,2,\dots\}$$ is
an orthonormal basis for the Hilbert space $L_2(\ss)$. Thus any
$f\in L_2(\ss)$ can be expressed by its Fourier (or Laplace)
series$$f=\sum_{\ell=0}^{\infty}
H_\ell(f)=\sum_{\ell=0}^{\infty}{\sum_{k=1}^{Z(d,\ell)} \lb
f,Y_{\ell,k}\rb Y_{\ell,k}},$$
 where $H_\ell^d(f)=\sum\limits_{k=1}^{Z(d,\ell)} \lb
f,Y_{\ell,k}\rb Y_{\ell,k},\ \ell=0,1,\dots,$ denote the
orthogonal  projections of $f$ onto $\mathcal{H}_\ell^d$, and
$$\lb f,Y_{\ell,k}\rb =\int_{\ss} f(x)
Y_{\ell,k}(x) \, d\sz(x)$$ are the  Fourier coefficients of $f$.
We have the following Parseval equality:
$$\|f|L_2(\ss)\|=\Big(\sum_{\ell=0}^{\infty}\sum_{k=1}^{Z(d,\ell)} |\lb
f,Y_{\ell,k}\rb |^2\Big)^{1/2}.$$

Let $\Lz=\{\lz_k\}_{k=0}^\infty$ be a bounded sequence, and let
$T^\Lz$ be a multiplier operator on $L_2(\ss)$ defined by $$T^\Lz
(f)=\sum_{\ell=0}^{\infty}\lz_\ell
H_\ell(f)=\sum_{\ell=0}^{\infty}\lz_\ell {\sum_{k=1}^{Z(d,\ell)}
\lb f,Y_{\ell,k}\rb Y_{\ell,k}}.$$

Let $SO(d+1)$ be the special rotation group of order ${d+1}$,
i.e., the set of all rotation on $\Bbb R^{d+1}$. For any $\rho \in
SO(d+1)$ and $f\in L_2(\ss)$, we define $\rho(f)(x)=f(\rho x)$.
  It is well known (see \cite[Proposition 2.2.9]{DaiX}) that
a bounded linear operator $T$ on $L_2(\ss)$ is  a multiplier
operator if and only if   $T\rho=\rho T$ for any $\rho \in
SO(d+1)$.

\begin{defn}\label{d2.1}Let $\Lz=\{\lz_k\}_{k=0}^\infty$ be  a non-increasing
positive sequence with $\lim\limits_{k\to\infty}\lz_k=0$. We
define the multiplier  space $H^\Lz(\ss)$ by
\begin{align*}H^\Lz(\ss)&:=\Big\{T^\Lz f\,    \big|\, f\in
L_2(\ss) \ {\rm and }\ \|T^\Lz
f\,\big|\,H^\Lz(\ss)\|=\|f|L_2(\ss)\|<\infty\Big\}\\&
 := \Big\{f\in
L_2(\ss)\,    \big|\,
\|f\,\big|\,H^\Lz(\ss)\|:=\Big(\sum_{\ell=0}^{\infty}\frac1{\lz_\ell^2}\sum_{k=1}^{Z(d,\ell)}
|\lb f,Y_{\ell,k}\rb |^2\Big)^{1/2}<\infty\Big\}.
\end{align*}\end{defn}

Clearly, the multiplier  space $H^\Lz(\ss)$ is a Hilbert space
with inner product
$$\langle f,g \rangle_{H^\Lz(\ss)}= \sum_{\ell=0}^{\infty}\frac1{\lz_\ell^2}\sum_{k=1}^{Z(d,\ell)}
\lb f,Y_{\ell,k}\rb\, \lb g,Y_{\ell,k}\rb.$$

We remark that Sobolev spaces and Gevrey type spaces on the sphere
$\ss$ are special multiplier spaces whose definitions are given as
follows.

\begin{defn}\label{d2.2}Let $r>0$ and $\square\in\{*,+,\#,-\}$. The  Sobolev space $H^{r,\square}(\ss)$ is the collection of all $f\in L_2(\ss)$ such
that
$$\|f\,\big|\,H^{r,\square}(\ss)\|=\Big(\sum_{\ell=0}^{\infty}(r_{\ell,d}^{\square})^{-2}\sum_{k=1}^{Z(d,\ell)}
|\lb f,Y_{\ell,k}\rb |^2\Big)^{1/2}<\infty,$$ where
\begin{align*}
r_{\ell,d}^{*}&=(1+(\ell(\ell+d-1))^r)^{-1/2},\ \
r_{\ell,d}^{+}=(1+\ell(\ell+d-1))^{-r/2},\\
r_{\ell,d}^{\#}&=(1+\ell)^{-r},\ \ \ \ \qquad \qquad\qquad
r_{\ell,d}^{-}=(\ell+(d-1)/2)^{-r},\ \ \ l=0,1,\dots.
\end{align*}
If we set $\Lz^{\square}=\{r_{k,d}^\square\}_{k=0}^\infty$, then
the Sobolev space $H^{r,\square}(\ss)$ is just the multiplier
space $H^{\Lz^\square}(\ss)$.
\end{defn}

\begin{rem}We note that the above four Sobolev norms are equivalent with the equivalence constants depending on $d$ and $r$.  The natural  Sobolev space on the sphere is
$H^{r,*}(\ss)$. The Sobolev space $H^{r,-}(\ss)$ is given in
\cite[Definition 3.23]{AH}.
\end{rem}

\begin{rem}Let $\triangle$ be the Laplace-Beltrami operator  on the sphere. It is well known
that the spaces $\mathcal{H}_\ell^d, \ \ell=0,1,2, \dots$, are
just the eigenspaces corresponding to the eigenvalues $-\ell(\ell
+d-1)$ of  the operator $\triangle$. Given $s>0$, we define the
$s$-th (fractional) order Laplace-Beltrami operator
$(-\triangle)^s$ on $\ss$ in a distributional sense by
$$H_0((-\triangle)^s(f))=0,\ \ \ H_\ell((-\triangle)^s(f))=(\ell(\ell
+d-1))^s H_\ell(f), \  \ \ell=1,2,\dots\ ,$$ where $f$ is a
distribution on $\ss$. We can define $(I-\triangle)^s, \
(((d-1)/2)^2I-\triangle)^s $ analogously, where $I$ is the
identity operator. For $r>0$ and $f\in H^{r}(\ss)$, we have the
following equalities.
\begin{align*}
\|f\,\big|\,H^{r,*}(\ss)\|&=\Big(\|f|L_2(\ss)\|^2+\|(-\triangle)^{r/2}f|L_2(\ss)\|^2\Big)^{1/2},\\
\|f\,\big|\,H^{r,+}(\ss)\|&=\Big(\|(I-\triangle)^{r/2}f|L_2(\ss)\|^2\Big)^{1/2},\\
\|f\,\big|\,H^{r,-}(\ss)\|&=\Big(\|(((d-1)/2)^2I-\triangle)^{r/2}f|L_2(\ss)\|^2\Big)^{1/2}.\end{align*}
\end{rem}

\begin{defn}\label{d2.3}Let $0<\az,\beta<\infty$. The Gevrey type space $G^{\az,\beta}(\ss)$ is the collection of all $f\in L_2(\ss)$ such
that
$$\|f\,\big|\,G^{\az,\beta}(\ss)\|=\Big(\sum_{\ell=0}^{\infty}e^{2\beta \ell^\az}\sum_{k=1}^{Z(d,\ell)}
|\lb f,Y_{\ell,k}\rb |^2\Big)^{1/2}<\infty.$$ If we set
$\Lz_{\az,\beta}=\{e^{-\beta k^\az}\}_{k=0}^\infty$, then the
Gevrey type space $G^{\az,\beta}(\ss)$ is just the multiplier
space $H^{\Lz_{\az,\beta}}(\ss)$. \end{defn}

\begin{rem} M. Gevrey \cite{G} introduced in 1918 the classes of smooth functions on $\Bbb R^d$ that are nowadays called Gevrey
classes. They
 have played an
important role in study of partial differential equation. A
standard reference on Gevrey spaces  is Rodino¡¯s book \cite{R}.
The Gevrey spaces on $d$-dimensional torus were introduced and
investigated in \cite{KMU}.  Our definition is a natural
generalization from $\Bbb R^d$ to $\ss$. When $d=1$, the unit ball
of $G^{\az,\beta}(\ss)$ recedes to the class $A_q^{\tau,b}$ with
$q=2, \ b=\az$ and $\tau=\beta$ introduced in \cite[p. 73]{Te}.
When $\alpha=1$, the action of the Gevrey type multiplier
 $T^{\Lz_{1,\beta}}$ on $f\in L_2(\ss)$ is  just the Possion integral of $f$ on the sphere (see \cite[pp. 34-35]{DaiX}). \end{rem}

\subsection{Approximation numbers}

\

 Let
${H}$ and ${G}$ be two  Hilbert spaces and $S$ be a compact linear
operator from  $H$ to $G$.  The fact concerning the approximation
numbers $a_n(S: H\to G)$  is well-known, see e.g. K\"onig
\cite[Section 1.b]{K}, Pinkus \cite[Theorem IV.2.2]{P}, and Novak
and Wo\'niakowski \cite[Corollary 4.12]{NW1}. The following lemma
is a simple  form  of the above fact.

Let $H$ be a separable Hilbert space, $\{e_k\}_{k=1}^\infty $ an
orthonormal basis  in $H$, and ${\tau}=\{\tau_k\}_{k=1}^\infty$ a
sequence of positive numbers with $\tau_1\ge\tau_2\ge \dots\ge
\tau_k\ge\cdots$. Let $H^{ \tau}$ be a Hilbert space defined by
$$H^\tau=\Big\{x\in H\ |\ \|x|H\|=\Big(\sum_{k=1}^\infty
\frac{|(x,e_k)|^2}{\tau_k^2}\Big)^{1/2}<\infty\Big\}.$$ According
to \cite[Corollary 2.6]{P} we have the following lemma.

\begin{lem}\label{l2.1}  Let $H, \tau $ and $H^{\tau }$ be defined as above.
Then $$a_n(I_d: H^\tau\to H)= \tau_n, \
 \ n\in\Bbb N_+.$$
\end{lem}

In the sequel, we always suppose that
$\Lz=\{\lz_{k,d}\}_{k=0}^\infty$ is  a non-increasing positive
sequence with $\lim\limits_{k\to\infty}\lz_{k,d}=0$. It follows
from Definition \ref{d2.1} that the mutiplier space $H^\Lz(\ss)$
is of form $H^\tau=(L_2(\ss))^\tau$ with
\begin{align*}
\{\tau_k\}_{k=0}^\infty=\{\lz_{0,d},\underbrace{\lz_{1,d},\cdots,\lz_{1,d}}_{Z(d,1)},\underbrace{\lz_{2,d},
\cdots,\lz_{2,d}}_{Z(d,2)},\cdots,\underbrace{\lz_{k,d},\cdots,\lz_{k,d}}_{Z(d,k)},\cdots\},
\end{align*} where $Z(d,m)$ and $C(d,m)$ are given in \eqref{2.1}
and \eqref{2.2}. According to Lemma \ref{l2.1} we obtain

\begin{thm} \label{t2.1} For $C(d,k-1)<n\le C(d,k),\ k=0,
1,2,\dots,$ we have $$a_n(T^\Lz: L_2(\ss)\to L_2(\ss))= a_n(I_d:
H^\Lz(\ss)\to L_2(\ss))=\lz_{k,d},
$$where we set $C(d,-1)=0$.

Specially, let $r>0$, $\square\in\{*,+,\#,-\}$, and
$0<\az,\beta<\infty$. Then for $C(d,k-1)<n\le C(d,k),\
k=0,1,2,\dots,$ we have
\begin{equation}\label{2.3} a_n(I_d:
H^{r,\square}(\ss)\to L_2(\ss))=r_{k,d}^\square,\end{equation} and
\begin{equation}\label{2.4} a_n(I_d: G^{\az,\beta}(\ss)\to
L_2(\ss))=e^{-\beta k^\az},\end{equation}where the definitions of
$H^{r,\square}(\ss),\ r_{k,d}^\square$, $\square\in\{*,+,\#,-\}$
are given in Definition \ref{d2.2}, and the definition of
$G^{\az,\beta}(\ss)$ is  in Definition \ref{d2.3}.
 \end{thm}

\section{Strong equivalences of approximation numbers}

This section is devoted to giving strong equivalence of the
approximation numbers of the Sobolev embeddings and the Gevrey
type embeddings on the sphere.

\begin{thm}\label{t3.1}Suppose that $\lim\limits_{k\to\infty}\lz_{k,d}k^s=1$ for some $s>0$. Then
\begin{equation}\label{3.1}\lim_{n\to\infty}n^{s/d}a_n(I_d: H^\Lz(\ss)\to L_2(\ss))=\Big(\frac2{d\,!}\Big)^{s/d},\end{equation}
Specially,  for $r>0$ and  $\square\in\{*,+,\#,-\}$, we have
\begin{equation}\label{3.2}\lim_{n\to\infty}n^{r/d}a_n(I_d:
H^{r,\square}(\ss)\to
L_2(\ss))=\Big(\frac2{d\,!}\Big)^{r/d}.\end{equation}
\end{thm}

\begin{proof}For $C(d,k-1)<n\le C(d,k),\ k=0,1,2,\dots,$ we have
$$a_n(I_d: H^\Lz(\ss)\to L_2(\ss))=\lz_{k,d}.$$
It follows that
\begin{align*}(C(d,k-1))^{s/d}\lz_{k,d}< n^{s/d}a_n(I_d: H^\Lz(\ss)\to L_2(\ss))\le (C(d,k))^{s/d}\lz_{k,d}.\end{align*}
We recall from \eqref{2.2} that $$C(d,k)=
\frac{(2k+d)(k+d-1)\,!}{k\,!\,d\,!},
$$which yields that
$$\lim_{k\to\infty}\frac{C(d,k)}{k^d}=\lim_{k\to\infty}\frac{C(d,k-1)}{k^d}=\frac2{d\,!}.$$
Since
\begin{align*}&\lim_{k\to\infty}(C(d,k-1))^{s/d}\lz_{k,d}=\lim_{k\to\infty}(C(d,k))^{s/d}\lz_{k,d}\\ &\quad\ =
\lim_{k\to\infty}\Big(\frac{C(d,k)}{k^d}\Big)^{s/d}k^s\lz_{k,d}=\Big(\frac2{d\,!}\Big)^{s/d},
\end{align*}
we obtain \eqref{3.1}. Theorem \ref{t3.1} is proved.
 \end{proof}

\begin{rem}   One can rephrase \eqref{3.2}  as  strong
equivalences $$a_n(I_d: H^{r,\square}(\ss)\to L_2(\ss))\sim
n^{-r/d}\Big(\frac2{d\,!}\Big)^{r/d}
$$for $r>0$ and  $\square\in\{*,+,\#,-\}$. The novelty of Theorems
\ref{t3.1}  is that they give strong equivalences of $a_n(I_d:
H^{r,\square}(\ss)\to L_2(\ss)) $ and provide asymptotically
optimal constants, for arbitrary fixed $d$ and $r>0$.
\end{rem}

\begin{thm}\label{t3.3}Let $0<\az<1$, $\beta>0$, and $\gamma=\Big(\frac2{d\,!}\Big)^{-\az/d}$. Then we have
\begin{equation}\label{3.3}\lim_{n\to \infty }e^{\beta \gz n^{\az/d}}a_n(I_d: G^{\az, \beta}(\ss)\to L_2(\ss))=1.\end{equation}\end{thm}
\begin{proof}It follows from \eqref{2.4} that for $C(d,k-1)<n\le C(d,k),\
k=0,1,2,\dots,$
\begin{equation*} a_n\equiv a_n(I_d: G^{\az,\beta}(\ss)\to
L_2(\ss))=e^{-\beta k^\az}.\end{equation*}Therefore, we have
\begin{equation}\label{3.4} e^{\beta \gz (C(d,k-1))^{\az/d}}e^{-\beta k^\az}<  e^{\beta \gz n^{\az/d}}a_n\le e^{\beta \gz (C(d,k))^{\az/d}} e^{-\beta k^\az}.\end{equation}
Since for $0<\az<1$,
$$\lim_{k\to\infty}(\gz (C(d,k))^{\az/d} -k^\az)=\lim_{k\to\infty}k^\az\Big(\Big( \big(1+\frac d{2k}\big)\prod_{j=1}^{d-1}(1+\frac jk)\Big)^{\frac \az d}-1\Big)=0,$$
we get
\begin{equation}\label{3.5}\lim_{n\to\infty}e^{\beta \gz (C(d,k))^{\az/d}} e^{-\beta
k^\az}=e^{\beta\lim\limits_{k\to\infty}(\gz (C(d,k))^{\az/d}
-k^\az)}=1.\end{equation} Similarly, we can show that
$$\lim_{n\to\infty}e^{\beta \gz (C(d,k-1))^{\az/d}} e^{-\beta
k^\az}=1,$$which combining with \eqref{3.5} and \eqref{3.4},
yields \eqref{3.3}. Theorem \ref{t3.3} is proved.
\end{proof}

\begin{rem}   One can rephrase \eqref{3.3} as  a strong
equivalence
$$a_n(I_d: G^{\az, \beta}(\ss)\to L_2(\ss))\sim e^{-\beta \gz
n^{\az/d}}$$ for $0<\az<1$ and $\beta>0$, where
$\gamma=\Big(\frac2{d\,!}\Big)^{-\az/d}$. The novelty of Theorems
\ref{t3.3} is that they give a strong equivalence of $ a_n(I_d:
G^{\az, \beta}(\ss)\to L_2(\ss))$ and provide asymptotically
optimal constants, for arbitrary fixed $d$, $0<\az<1$, and
$\beta>0$.
\end{rem}

\begin{rem}For  $\az=1$,  we have
$$\lim_{n\to\infty}e^{\beta \gz (C(d,k-1))^{\az/d}} e^{-\beta
k^\az}=e^{ \frac{\beta(d-1)^2}{2d}}\neq e^{ \frac{\beta d}{2}}=
\lim_{n\to\infty}e^{\beta \gz (C(d,k))^{\az/d}} e^{-\beta k^\az},
$$
which means that the strong equivalence $$a_n(I_d: G^{\az,
\beta}(\ss)\to L_2(\ss))\sim e^{-\beta \gz n^{\az/d}}$$ does not
hold. However, we have the weak equivalence
\begin{equation}\label{3.6}a_n(I_d: G^{\az, \beta}(\ss)\to L_2(\ss))\asymp e^{-\beta
(nd\,!/2)^{1/d}}, \end{equation}  where the equivalence constants
may depend  on $d$, but not on $n$.

For $\az>1$, there seems even no weak asymptotics of $a_n(I_d:
G^{\az, \beta}(\ss)\to L_2(\ss))$ as \eqref{3.6}.
\end{rem}

 \section{Preasymptotics and asymptotics of the approximation numbers}

 This section is devoted to giving preasymptotics and asymptotics  of the
approximation numbers of the Sobolev embeddings and the Gevrey
type embeddings on the sphere.

 \begin{lem}\label{l4.1} For  $m\in \Bbb N $ and $d\in\Bbb  N_+$ we have
\begin{align}\label{4.1}\max\Big\{\big(1+\frac{m}{d}\big)^{d},\,\big(1+\frac{d}{m}\big)^{m}\Big\}\leq
C(d,m)\le
\min\Big\{e^d\big(1+\frac{m}{d}\big)^{d},\,e^m\big(1+\frac{d}{m}\big)^{m}\Big\}.
\end{align}
\end{lem}
\begin{proof} We note  that $$\binom{m+d}d\le C(d,m)=\binom{m+d}d \frac{2m+d}{m+d}\le
2\binom{m+d}d.
$$
Using the inequality (see \cite[(3.6)]{KSU1})
\begin{align*}\label{4.2} \binom{m+d}{d}\leq
e^{d-1}{(1+\frac{m}{d})}^{d},
\end{align*}we get the upper estimate of $C(d,m)$. Using the inequality (see
\cite[(3.5)]{KSU1})$$\max\Big\{\big(1+\frac{m}{d}\big)^{d},\,\big(1+\frac{d}{m}\big)^{m}\Big\}\leq
\binom{m+d}d,
$$we get the lower estimate of $C(d,m)$. Lemma
\ref{l4.1} is proved.
\end{proof}

\begin{thm}\label{t4.1} Let $r>0$.
We have
\begin{equation}\label{4.2}
a_n(I_d: H^{r,*}(\ss)\to L_2(\ss)) \asymp
\left\{\begin{matrix} 1,\ \ \ &n=1,\\
d^{-r/2}, & \ \  2\le n\leq d,\\
 d^{-r/2}\Big(\frac{\log(1+\frac{d}{\log n})}{\log n}\Big)^{r/2},&\ \  d\le n\le 2^d, \\
 d^{-r}n^{-r/d},&\ \  n\ge 2^d,
\end{matrix}\right.
\end{equation}
where the equivalence constants depend only on $r$, but not on $d$
and $n$.
\end{thm}
\begin{proof}  We
have for $n=1$,
$$a_n(I_d: H^{r,*}(\ss)\to L_2(\ss)) =1.$$

For $2\le n\le C(d,1)=d+2$, we have
$$   a_n(I_d: H^{r,*}(\ss)\rightarrow L_2(\ss))=r^*_{1,d}=(1+d^r)^{-1/2}\asymp d^{-r/2}.$$This means that
$$a_n(I_d: H^{r,*}(\ss)\rightarrow L_2(\ss))\asymp d^{-r/2}\asymp \left\{\begin{matrix}
d^{-r/2}, &   2\le n\leq d,\\
 d^{-r/2}\Big(\frac{\log(1+\frac{d}{\log n})}{\log n}\Big)^{r/2},& d\le n\le
 C(d,1).
\end{matrix}\right.$$

For $C(d,m-1)<n\le C(d,m),\ 2\le m\le d$, we have
\begin{equation}\label{4.3}a_n(I_d:  H^{r,*}(\ss)\rightarrow L_2(
\ss))=(1+(m(m+d-1))^r)^{-1/2}\asymp
m^{-r/2}d^{-r/2}.\end{equation} By \eqref{4.1} we get that
$$n\le C(d,m)\le e^m(1+d/m)^m\le \Big(\frac{2ed}m\Big)^m.$$
 It follows that $$ \log n\le m\log
(2ed/m), $$ which implies
\begin{equation}\label{4.4} m\ge \frac{\log n}{\log (2ed/m)}\end{equation}
and $$\log \Big(\frac{2ed}{\log n}\Big)\ge
\log\Big(\frac{2ed}{m\log
(2ed/m)}\Big)=\log(\frac{2ed}m)-\log\Big(\log(\frac {2ed}m)\Big).
$$
 Using the inequality  $x\ge
2\log x$ for $x\ge 2$, we obtain
$$\log \Big(\frac{2ed}{\log n}\Big)\ge \frac12\log(\frac{2ed}m).
$$
This combining with \eqref{4.4} yields
\begin{equation}\label{4.5} m\ge \frac{\log n}{2\log
(2ed/(\log n))}.\end{equation} On other hand, it follows from
\eqref{4.1} that
$$n>C(d,m-1)\ge \Big(1+\frac d{m-1}\Big)^{m-1}.$$ This yields
$$m-1\le\frac{ \log n}{\log \big(1+\frac {d}{m-1}\big)}\le\log n.$$
It follows that
$$m\le \frac{ \log n}{\log \big(\frac {2d}{\log n}\big)}+1,$$
which combining with \eqref{4.5}, leads to
\begin{equation}\label{4.6} m\asymp \frac{ \log n}{1+\log
\big(\frac {d}{\log n}\big)}.\end{equation}It follows from
\eqref{4.3} that for $C(d,m-1)<n\le C(d,m),\ 2\le m\le d$,
\begin{align*} a_n(I_d:\  H^{r,*}(\ss)\rightarrow L_2( \ss))\asymp
m^{-r/2}d^{-r/2}\asymp d^{-r/2}\Big(\frac{\log(1+\frac{d}{\log
n})}{\log n}\Big)^{r/2}.\end{align*} Note that
$$  2^d\le C(d,d)\le (2e)^d$$ and for $2^d\le n\le (2e)^d$,  $$ \frac{\log n} {\log(1+\frac{d}{\log
n})}\asymp d\ \ {\rm and}\ \ n^{-r/d}\asymp 1.$$ We obtain that
for $C(d,1)<n\le C(d,d)$,
\begin{align*} a_n(I_d:\  H^{r,*}(\ss)\rightarrow L_2( \ss))&\asymp d^{-r/2}\Big(\frac{\log(1+\frac{d}{\log
n})}{\log n}\Big)^{r/2}\\ &\asymp \left\{\begin{matrix}
 d^{-r/2}\Big(\frac{\log(1+\frac{d}{\log n})}{\log n}\Big)^{r/2},&\ \  d\le n\le 2^d, \\
 d^{-r}n^{-r/d},&\ \  2^d\le n\le C(d,d).
\end{matrix}\right..\end{align*}

For $C(d,m-1)<n\le C(d,m),\ m>d$, we have
\begin{equation}\label{4.7}a_n(I_d:\, H^{r,*}(\ss)\rightarrow L_2(
\ss))=(1+(m(m+d-1))^r)^{-1/2}\asymp m^{-r}.\end{equation} It
follows from \eqref{4.1} that
\begin{align*}\frac md\le 1+\frac{m-1}d \le (C(d,m-1))^{\frac 1d}\le  n^{\frac 1d}\le (C(d,m))^{\frac 1d}\le e(1+\frac md)\le \frac{2e m}d. \end{align*}
Therefore, we get \begin{equation}\label{4.8}m\asymp
dn^{1/d},\end{equation} which combining with \eqref{4.7} yields
that for $n>C(d,d)$,
$$a_n(I_d:\,  H^{r,*}(\ss)\rightarrow L_2(
\ss))\asymp d^{-r}n^{-r/d}.$$The proof of Theorem \ref{t4.1} is
complete.
\end{proof}

Using the same method as in the proof of  Theorem \ref{t4.1}, we
can obtain the following two theorems.

\begin{thm}\label{t4.2} Let $r>0$.
We have
\begin{equation}\label{4.9}
a_n(I_d: H^{r,+}(\ss)\to L_2(\ss)) \asymp
\left\{\begin{matrix} 1,\ \ \ &n=1,\\
d^{-r/2}, & \ \  2\le n\leq d,\\
 d^{-r/2}\Big(\frac{\log(1+\frac{d}{\log n})}{\log n}\Big)^{r/2},&\ \  d\le n\le 2^d, \\
 d^{-r}n^{-r/d},&\ \  n\ge 2^d,
\end{matrix}\right.
\end{equation}

\begin{equation}\label{4.10}
a_n(I_d: H^{r,\#}(\ss)\to L_2(\ss)) \asymp \left\{\begin{matrix}
1, & \ \  1\le n\leq d,\\
 \Big(\frac{\log(1+\frac{d}{\log n})}{\log n}\Big)^{r},&\ \  d\le n\le 2^d, \\
 d^{-r}n^{-r/d},&\ \  n\ge 2^d,
\end{matrix}\right.
\end{equation}
and
\begin{equation}\label{4.11}
a_n(I_d: H^{r,-}(\ss)\to L_2(\ss)) \asymp \left\{\begin{matrix}
 d^{-r},&\ \  1\le n\le 2^d, \\
 d^{-r}n^{-r/d},&\ \  n\ge 2^d,
\end{matrix}\right.
\end{equation}
 where all above equivalence constants depend only on $r$, but not on
$d$ and $n$.
\end{thm}

\begin{thm}\label{t4.3} Let $\alpha,\beta>0$.
We have
\begin{equation}\label{4.12}
\ln \big(a_n(I_d: G^{\az,\beta}(\ss)\to L_2(\ss))\big) \asymp
-\beta
 \left\{\begin{matrix}
1, & \ \  1\le n\leq d,\\
 \Big(\frac{\log n}{\log(1+\frac{d}{\log n})}\Big)^{\az},&\ \  d\le n\le 2^d, \\
 d^{\az}n^{\az/d},&\ \  n\ge 2^d,
\end{matrix}\right.
\end{equation}
 where the equivalence constants depend only on $\az$, but not on
$d$ and $n$.
\end{thm}

\section{Tractability analysis}

Recently, there has been an increasing interest in $d$-variate
computational problems  with large or even huge $d$. Such problems
are usually solved by algorithms that use finitely many
information operations. In this paper, we limit ourselves to the
worst case setting, information operation is defined as the
evaluation of a continuous linear functional, and we deal with a
Hilbert space setting (source and target space).  The information
complexity $n(\vz, d)$ is defined as the minimal number of
information operations which are needed to find an approximating
solution to  within an error threshold $\vz$. A central issue is
the study of how the information complexity depends on $\vz^{-1}$
and $d$. Such problem is called the tractable problem. Nowadays
tractability of multivariate problems is a very active research
area (see \cite{NW1, NW2,NW3} and the references therein).

 Let
${H_d}$ and ${G_d}$ be two sequences of Hilbert spaces and for
each $d\in \Bbb N_+$,  $F_d$ be the unit ball of  $H_d$. Assume a
sequence of bounded linear operators (solution operators)
$$S_d : H_d\rightarrow G_d$$ for all $d \in \Bbb N_+$.
For $n\in \Bbb N_+$ and $f\in F_d$, $S_d f$ can be approximated by
algorithms
$$A_{n,d}(f)=\Phi _{n,d}(L_1(f),...,L_n(f)),$$
where  $L_j,\ j=1,\dots,n$ are continuous linear functionals on
$F_d$ which are called  general information, and $\Phi _{n,d} :
\Bbb R^n\rightarrow G_d$ is an arbitrary mapping. The worst case
error $e(A_{n,d})$ of the algorithm $A_{n,d}$ is defined as
$$e(A_{n,d})=\sup_{f\in F_d}
\|S_d(f)-A_{n,d}(f)\|_{G_d}.$$  Furthermore, we define the $n$th
minimal worst-case error as
$$e(n,d )=\inf_{A_{n,d}}e(A_{n,d}),$$
where the infimum is taken over all algorithms using $n$
information operators $L_1,L_2,...,L_n$. For $n=0$, we use
$A_{0,d}=0$. The error of $A_{0,d}$ is called the initial error
and is given by
$$e(0,d )=\sup_{f\in F_d}\|S_d f\|_{G_d}.$$
The $n$th minimal worst-case error $e(n,d)$ with respect to
arbitrary  algorithms and general information in the Hilbert
setting is just the $n+1$-approximation number $a_{n+1}(S_d:H_d\to
G_d)$ (see \cite[p. 118]{NW1}), i.e.,
$$e(n,d)=a_{n+1}(S_d:H_d\to G_d).$$

    In this paper, we consider the embedding operators $S_d=I_d$
(formal identity operators).
  For $\varepsilon \in (0,1)$ and $d\in \Bbb N_+$, let $n(\varepsilon, d)$ be the information
  complexity which is defined as the minimal number of continuous linear functionals which are necessary to obtain
  an $\varepsilon -$approximation of $I_d$, i.e.,
$$n(\varepsilon ,d)=\min\{n\,|\,e(n,d)\leq \varepsilon CRI_d \},$$
where
$$CRI_d=\begin{cases}
1,&\  \   \text{for the absolute error criterion},\\
e(0,d), &\   \  \text{for the normalized error
criterion}.\end{cases}$$

 Next, we list the concepts of tractability below. We say
that the approximation problem is

 $\bullet$ \textbf {weakly tractable}, if
\begin{equation*}\label{t1.2}
\lim_{\varepsilon ^{-1}+d\rightarrow \infty }\frac{\ln
n(\varepsilon ,d)}{\varepsilon ^{-1}+d}=0.
\end{equation*}Otherwise, the approximation problem is called intractable.

 $\bullet$  \textbf {uniformly weakly tractable}, if for all $s,t>0$
\begin{equation*}\label{t1.3}
\lim_{\varepsilon ^{-1}+d\rightarrow \infty }\frac{\ln
n(\varepsilon ,d)}{(\varepsilon ^{-1})^{s }+d^{t }}=0.
\end{equation*}

 $\bullet$  \textbf {quasi-polynomially tractable}, if there
exist two positive constants $C,\, t$ such that for all $d\in \Bbb
N_+, \ \varepsilon
 \in(0,1)$,
\begin{equation}\label{t1.4}
n(\varepsilon ,d)\leq C\exp(t(1+\ln\varepsilon ^{-1})(1+\ln d)).
\end{equation}
The infimum of $t$ satisfying \eqref{t1.4} is called the exponent
of quasi-polynomial tractability and is denoted by $t^{\rm qpol}$.

 $\bullet$ \textbf {polynomially tractable}, if there exist non-negative
 numbers $C, p$ and $q$ such that for all $d\in \Bbb N_+, \ \varepsilon
 \in(0,1)$,
\begin{equation}\label{t1.5}
n(\varepsilon ,d )\leq Cd^q(\varepsilon ^{-1})^p.
\end{equation}

 $\bullet$  \textbf {strongly polynomially tractable},
 if there exist non-negative numbers $C$ and $p$ such that for all $d\in \Bbb N_+,\ \varepsilon \in
 (0,1)$,
\begin{equation*}\label{t1.6}
n(\varepsilon ,d )\leq C(\varepsilon ^{-1})^p.
\end{equation*}

Of course, the latter  tractability implies the former
tractability.

 $\bullet$ The approximation problem
suffers from \textbf {the curse of dimensionality}, if there exist
positive numbers $C, \varepsilon _0, \gamma $ such that for all
$0<\varepsilon\leq \varepsilon _{0}$ and infinitely many $d\in
\Bbb N_+$,
\begin{equation*}\label{1.7}
n(\varepsilon ,d)\geq C(1+\gamma )^{d}.
\end{equation*}

 Recently, Siedlecki and  Weimar introduced the
notion of $(s,t)$-weak tractability in \cite{SW}. If for some
fixed $s,\ t>0$ it holds
\begin{equation*}
\lim_{\varepsilon ^{-1}+d\rightarrow \infty }\frac{\ln
n(\varepsilon ,d)}{(\varepsilon ^{-1})^{s}+d^{t}}=0,
\end{equation*}
then the approximation problem is called $(s, t)$-weakly
tractable. Clearly,  $(1,1)$-weak tractability is  just weak
tractability,  whereas the approximation problem is uniformly
weakly tractable  if it is $(s,t)$-weakly tractable  for all
positive $s$ and $t$ (see \cite{S}). Also, if the approximation
problem suffers from the curse of dimensionality, then for any
$s>0, \ 0<t\le1,$ it is not $(s,t)$-weakly tractable.

We introduce the following lemma which is used in the proofs of
main results.

\begin{lem}\label{l5.1} Let $m,d\in\Bbb N_+$, $s>0,\  t>1$. Then
we have
$$\lim_{m+d\to\infty}\frac{m\ln(m+d)}{m^t+d^s}=\lim_{m+d\to\infty}\frac{d\ln(m+d)}{m^s+d^t}=0.$$
\end{lem}
{\begin{proof}We set $\gamma=s/t$. Let $x=m+d^{\gamma}$. Then
$m+d\to\infty$ if and only if $x\to\infty$. If $\gamma\ge 1$, then
$m+d\le x$, and if $\gamma <1$, then $$m+d\le m^{1/\gz}+d\le
(m+d^\gamma)^{1/\gz}=x^{1/\gz}.$$It follows that
$$m\ln(m+d)\le \max\{1,1/\gz\} x\ln x.$$Using the inequality
$$(a+b)^p\le 2^p(a^p+b^p),\ a,b\ge0, p>0,$$ we get
$$m^t+d^s\ge 2^{-t}(m+d^{s/t})^t=2^{-t}x^t.$$
Hence, we have
$$0\le \lim_{m+d\to\infty}\frac{m\ln(m+d)}{m^t+d^s}\le\lim_{x\to\infty}\frac{\max\{1,1/\gz\} x\ln
x}{2^{-t}x^t}=0,$$which means that $$
\lim_{m+d\to\infty}\frac{m\ln(m+d)}{m^t+d^s}=0. $$

Similarly, we can  prove
$$\lim\limits_{m+d\to\infty}\frac{d\ln(m+d)}{m^s+d^t}=0.$$ Lemma
\ref{l5.1} is proved.
\end{proof}

We remark that if  $e(0,d)=\|I_d\|=1$, then the normalized error
criterion and the absolute error criterion coincide.
 We emphasize that $$e(0,d)=1$$ for
the approximation problems $$I_d: H^{r,\square}(\ss)\rightarrow
L_2(\ss),\ \ r>0,\  \square\in \{*,+,\#\},$$ and $$ I_d:
G^{\az,\beta}(\ss)\rightarrow L_2(\ss),\ \ \az,\beta>0.$$ However,
we have $$e(0,d)=\big(\frac{d-1}2\big)^{-r}$$ for  the
approximation problems $$I_d: H^{r,-}(\ss)\rightarrow L_2(\ss),\ \
r>0.$$

\begin{thm}\label{t5.1}
Let $r>0$ and  $s,\,t>0$. Then

(1) the approximation problems
$$I_d: H^{r,\square}(\ss)\rightarrow L_2(\ss), \ \ \square\in \{*,+,\#\}$$ and for the absolute error criterion  the approximation problem
$$I_d: H^{r,-}(\ss)\rightarrow L_2(\ss)$$ are $(s,
t)$-weakly tractable if and only if $r>1/s$ and $t>0$ or $s>0$ and
$t>1$.

Specially, for the absolute error criterion the approximation
problems $I_d: H^{r,\square}(\ss)\rightarrow L_2(\ss), \
\square\in \{*,+,\#,-\}$
 are weakly tractable if and only if $r>1$, not uniformly
weakly tractable,  and do not suffer from the curse of
dimensionality.

(2) for the normalized error criterion,  the approximation problem
$$I_d: H^{r,-}(\ss)\rightarrow L_2(\ss)$$ suffers from the curse of
dimensionality.
\end{thm}

 \begin{proof} (1) First we show that for the absolute error criterion
 the approximation problems $I_d: H^{r,\square}(\ss)\rightarrow L_2(\ss)$ are not $(s, t)$-weakly
tractable if $0<s\le 1/r$ and $0<t\le1$, where $r>0$,\ $\square\in
\{*,+,\#,-\}$.

We note that
$$e(n,d)=a_{n+1}(I_d:  H^{r,\square}(\ss)\rightarrow L_2(\ss)).$$It
follows from Theorem  \ref{t2.1} that for $C(d,m-1)<n+1\le
C(d,m),\ m\ge1 $, we have
\begin{equation*}e(n,d)=r_{m,d}^{\square},\end{equation*}where
$r_{m,d}^{\square}$ are given in Definition \ref{d2.2}. This
implies that
$$e(C(d,d)-1,d)=r_{d,d}^\square.$$
 Choose
$$\vz=\vz_d=r_{d+1,d}^\square<r_{d,d}^\square.$$  Then $$n(\vz_d,d)=\inf\{n\in\Bbb N\ |\ e(n,d)\le \vz_d\}\ge C(d,d)\ge 2^d.$$If $0<s\le 1/r$ and $0<t\le1$,
then we have $$\lim_{1/\vz_d+d\to\infty}\frac{\ln
(n(\vz_d,d))}{(\vz_d)^{-s}+d^t}\ge \lim_{d\to\infty}\frac
{d}{(r_{d+1,d}^\square)^{-1/r}+d}\neq0,$$ which implies that for
the absolute error criterion the approximation problems $I_d:
H^{r,\square}(\ss)\rightarrow L_2(\ss)$ are not $(s, t)$-weakly
tractable if $0<s\le 1/r$ and $0<t\le1$.

Next we show that if $s>1/r$ and $t>0$ or $s>0$ and $t>1$, then
for the absolute error criterion the approximation problems $I_d:
H^{r,\square}(\ss) \rightarrow L_2(\ss)$ are $(s, t)$-weakly
tractable.

Let $0<\vz<1$ be given and select $m\in\Bbb N_+$ such that
$$r_{m,d}^\square\le\vz< r_{m-1,d}^\square.$$Since
\begin{equation*}e(n,d)=r_{m,d}^\square\end{equation*}
for $C(d,m-1)<n+1\le C(d,m)$, we get $$n(\vz,d)\le
n(r_{m,d}^\square,d)=\inf\{n\in\Bbb N\ |\ e(n,d)\le
r_{m,d}^\square\}=C(d,m-1)< C(d,m).$$ For $s>0$ and $t>1$, it
follows from \eqref{4.1} and Lemma \ref{l5.1} that
$$\lim_{\frac 1\vz+d\to\infty}\frac{\ln n(\vz,d)}{\vz^{-s}+d^t}\le \lim_{m+d\to\infty}\frac {d(1+\ln(\frac{m+d}d))}{(r_{m-1,d}^\square)^{-s}+d^t}\le
\lim_{m+d\to\infty}\frac {d(1+\ln({m+d}))}{(m-1)^{rs}+d^t}=0.$$
For $s>1/r$ and $t>0$, it follows from \eqref{4.1} and Lemma
\ref{l5.1} that
$$\lim_{\frac 1\vz+d\to\infty}\frac{\ln n(\vz,d)}{\vz^{-s}+d^t}\le \lim_{m+d\to\infty}\frac {m(1+\ln(\frac{m+d}m))}{(r_{m-1,d}^\square)^{-s}+d^t}\le
\lim_{m+d\to\infty}\frac {m(1+\ln({m+d}))}{(m-1)^{rs}+d^t}=0.$$
Hence, for the absolute error criterion the approximation problems
$I_d: H^{r,\square}(\ss) \rightarrow L_2(\ss)$ are $(s, t)$-weakly
tractable if and only if  $s>1/r$ and $t>0$ or $s>0$ and $t>1$.

(2) It follows from Theorem \ref{t4.2} that for $0\le n\le 2^d$,
$$e(n,d)=a_{n+1}(I_d:  H^{r,-}(\ss)\rightarrow
L_2(\ss))\asymp d^{-r}.$$ This means that there is a positive
constant $c$ depending only on $r$ such that
$$e(2^d,d)\ge c\, e(0,d).$$
Choose $\vz\in(0,c)$. Then for the normalized criterion we have
$$n(\vz,d)=\min\{n\,|\,e(n,d)\leq \varepsilon\, e(0,d)\}\ge 2^d.$$ Hence, for the normalized error criterion the approximation problem
$I_d: H^{r,-}(\ss)\rightarrow L_2(\ss)$ suffers from the curse of
dimensionality.

 Theorem \ref{t5.1} is proved.
\end{proof}

\begin{thm}\label{t5.2}
Let $\az,\beta>0$. Then
 the approximation problem
$$I_d: G^{\az,\beta}(\ss)\rightarrow L_2(\ss)$$

(1)  is uniformly   weakly tractable.

(2)   is not polynomially tractable.

 (3)   is quasi-polynomially tractable if and only if $\az\ge 1$ and  the exponent
of quasi-polynomial tractability is $$t^{\rm qpol}=\sup_{m\in \Bbb
N}\frac m {1+\beta m^\az},\ \ \az\ge1.$$Specially, if $\az=1$,
then $t^{\rm qpol}=\frac1\beta$.

\end{thm}
\begin{proof}  (1) For $0<\vz<1$, we choose $m\in\Bbb N$ such that $$e^{-\beta(m+1)^{\az}}\le \vz<e^{-\beta m^{\az}}.$$Since for $C(d,m)<n+1\le
C(d,m+1),\ m\in\Bbb N $,
\begin{equation*}e(n,d)=a_{n+1}(I_d:G^{\az,\beta}(\ss)\to L_2(\ss))=e^{-\beta(m+1)^{\az}},\end{equation*}we get
$$n(\vz,d)\le n(e^{-\beta(m+1)^{\az}}, d)=C(d,m).$$
For any $s,t>0$, by \eqref{4.1} we have
$$\frac{\ln(n(\vz,d))}{\vz^{-s}+d^t}\le
\frac{\ln(C(d,m))}{e^{s\beta m^\az}+d^t}\le
\frac{m(1+\ln(m+d))}{e^{s\beta m^\az}+d^t}.$$ We note that
$$\lim_{m\to\infty}\frac{m^2}{e^{s\beta
m^\az}}=\lim_{x\to+\infty}\frac{x^{2/\az}}{e^{s\beta x}}=0.$$ This
implies that there exists a positive constant $c$ such that
$$e^{s\beta m^\az}\ge c m^2. $$It follows from Lemma \ref{l5.1} that
$$\lim_{\frac1\vz+d\to\infty}\frac{\ln(n(\vz,d))}{\vz^{-s}+d^t}\le
\lim_{m+d\to\infty}\frac {m\ln(m+d)}{cm^2+d^t}=0.$$ Hence, the
approximation problem $I_d: G^{\az,\beta}(\ss)\rightarrow
L_2(\ss)$   is uniformly   weakly tractable for any $\az,\beta>0$.

(2) For any $p,q>0$, we choose $$\vz_d=e^{-\beta (m_d+1)^\az},\ \
m_d=[(\frac q{\beta p}\ln d)^{1/\az}]-1,$$where $[x]$ denotes the
largest integer not exceeding $x\in\Bbb R$. Such selection is to
make $$(\vz_d)^{-p}\le d^q.$$ Since for $C(d,m_d)<n+1\le
C(d,m_d+1),\ m\in\Bbb N $,
\begin{equation*}e(n,d)=a_{n+1}(I_d:G^{\az,\beta}(\ss)\to L_2(\ss))=\vz_d,\end{equation*}we
get
$$n(\vz_d,d)=C(d,m_d)\ge (1+\frac d{m_d})^{m_d}\ge (\frac d{m_d})^{m_d}=e^{m_d(\ln d-\ln m_d)} .$$
Since $\lim\limits_{d\to\infty}m_d=+\infty$ and
$\lim\limits_{d\to\infty}\frac{\ln m_d}{\ln d}=0$, we get for
sufficiently large $d$, $$\ln d-\ln m_d\ge \frac12\ln d\ \ {\rm
and}\ \ \frac {m_d}2\ge 2q+1.$$ It follows that
$$\lim_{d\to\infty}\frac{n(\vz_d,d)}{(\vz_d)^{-p}d^q}\ge
\lim_{d\to\infty}\frac{e^{(2q+1)\ln
d}}{d^{2q}}=\lim_{d\to\infty}d= +\infty,$$which implies that
\eqref{t1.5} does not hold. Hence, the approximation problem $I_d:
G^{\az,\beta}(\ss)\rightarrow L_2(\ss)$   is not polynomially
tractable for any $\az,\beta>0$.

(3) Let $0<\az<1$ and $\beta>0$. We choose $\vz_d$ and $m_d$ as
above, i.e.,
$$\vz_d=e^{-\beta (m_d+1)^\az},\ \
m_d=[(\frac q{\beta p}\ln d)^{1/\az}]-1.$$ Then for sufficiently
large $d$, $$n(\vz_d,d)=C(d,m_d)\ge e^{\frac12m_d\ln
d}=d^{m_d/2}.$$ It follows that for any $t>0$,
$$\lim_{d\to\infty}\frac{n(\vz_d,d)}{e^{t(1+\ln d)(1+\ln
\frac 1{\vz_d})}}\ge \lim_{d\to\infty}\frac{d^{m_d/2}}{e^{4t
2^\az\beta (m_d)^\az\ln d}}=\lim_{d\to\infty}d^{\frac{m_d}2-42^\az
t\beta (m_d)^\az}=+\infty,
$$which means that \eqref{t1.4} is not valid.
Hence,  the approximation problem $I_d:
G^{\az,\beta}(\ss)\rightarrow L_2(\ss)$   is not
quasi-polynomially tractable if  $0<\az<1$.

Let $\az\ge 1$ and $\beta>0$. We set
\begin{equation}\label{5.2}t_0=\sup_{m\in \Bbb N}\frac m {1+\beta
m^\az}.\end{equation}
 For any $0<\vz<1$, we choose
$m\in\Bbb N$ such that $$e^{-\beta(m+1)^{\az}}\le \vz<e^{-\beta
m^{\az}}.$$Then we have
$$n(\vz,d)\le n(e^{-\beta(m+1)^{\az}}, d)=C(d,m).$$
We note from \eqref{5.2} that $$\sup_{d\in\Bbb N_+,\
\vz\in(0,1)}\frac{n(\vz,d)}{e^{t_0(1+\ln d)(1+\ln \frac
1{\vz})}}\le \sup_{d\in\Bbb N_+,\ m\in \Bbb N}
\frac{C(d,m)}{(ed)^{t_0(1+\beta m^\az)}}\le \sup_{d\in\Bbb N_+,\
m\in \Bbb N} \frac{C(d,m)}{(ed)^m}.
$$
From \eqref{2.2} we get $$C(d,0)=1,\ C(d,1)=d+2,\ C(1,m)=2m+1,$$
which yields
$$\sup_{d\in\Bbb N_+,\ m=0,1} \frac{C(d,m)}{(ed)^m}\le 2 \ \ {\rm
and}\ \ \sup_{m\in \Bbb N,\ d=1} \frac{C(d,m)}{(ed)^m}\le 2.$$ If
$m,d\ge 2$, then by \eqref{4.1} we get
$$\frac{C(d,m)}{(ed)^m}\le
\frac{e^m(1+d/m)^m}{(ed)^m}=(\frac1d+\frac1m)^m\le1.$$ This means
that for any $d\in\Bbb N_+$ and any $\vz\in(0,1)$,
$$n(\vz,d)\le 2e^{t_0(1+\ln d)(1+\ln \frac
1{\vz})}.$$Hence, the approximation problem $I_d:
G^{\az,\beta}(\ss)\rightarrow L_2(\ss)$   is  quasi-polynomially
tractable if  $\az\ge 1$ and the exponent $t^{\rm qpol}$ of
quasi-polynomial tractability satisfies $$t^{\rm qpol}\le
t_0=\sup_{m\in \Bbb N}\frac m {1+\beta m^\az}.$$

Let  $\az=1$. Then $t_0=1/\beta$. For $t<t_0=1/\beta$, we  choose
$\vz_d$ and $m_d$ as above. Such $m_d$ satisfies
$\lim\limits_{d\to\infty}m_d=+\infty$ and  $\lim\limits_{d\to
\infty}\frac{\ln m_d}{\ln d}=0$. Then $n(\vz_d,d)=C(d,m_d)$. It
follows that
\begin{align*}\lim_{d\to \infty}\frac{n(\vz_d,d)}{e^{t(1+\ln
d)(1+\ln \frac1{\vz_d})}}&=\lim_{d\to
\infty}\frac{C(d,m_d)}{e^{t(1+\ln d)(1+\beta
(m_d+1))}}\\
&\ge \lim_{d\to \infty}\frac{(1+d/m_d)^{m_d}}{(ed)^{t+t\beta
(m_d+1)}}.
\end{align*}
Let $\gz$ be such that $\gz\in (t\beta,1)$. Then for  sufficiently
large $d$, we have
$$\ln d-\ln m_d\ge \gamma \ln d,$$ which  means that for  sufficiently large
$d$, $$(1+d/m_d)^{m_d}\ge e^{m_d(\ln d-\ln m_d)}\ge e^{\gamma
m_d\ln d}=d^{\gamma m_d}.$$ Using the facts $$a^b=d^{b\frac{\ln
a}{\ln d}},\ \ \gz-t\beta>0, \ \ {\rm and}\ \
\lim_{d\to\infty}m_d=+\infty,$$we get
\begin{align*}\lim_{d\to \infty}\frac{n(\vz_d,d)}{e^{t(1+\ln
d)(1+\ln\frac1{\vz_d})}}&\ge\lim_{d\to\infty} \frac{d^{\gamma
m_d}}{(ed)^{t+t\beta (m_d+1)}}\\
&=\lim_{d\to\infty}d^{(\gz-t\beta-\frac{t\beta}{\ln
d})m_d-(t+t\beta)(1+\frac1{\ln d})}=+ \infty,
\end{align*}which means that \eqref{t1.4} is not true with
$t<t_0=1/\beta$. Hence, we have for $\az=1$,
$$t^{\rm qpol}=t_0=1/\beta.$$

Let $\az>1$. Since $\lim\limits_{m\to\infty}\frac m{1+\beta
m^\az}=0$, there exists a positive integer  $m_0$ depending only
on $\az>1$ and $\beta>0$ such that
$$t_0=\sup_{m\in \Bbb N}\frac m {1+\beta m^\az}=\frac {m_0} {1+\beta
(m_0)^\az}.$$ We choose $\vz_d\in (0,1)$ such that
$$e^{-\beta(m_0+1)^{\az}}\le \vz_d<e^{-\beta (m_0)^{\az}}\ \  {\rm
and}\ \ \lim_{d\to\infty}\vz_d=e^{-\beta (m_0)^{\az}}.$$ Since for
$C(d,m)<n+1\le C(d,m+1),\ m\in\Bbb N $,
\begin{equation*}e(n,d)=a_{n+1}(I_d:G^{\az,\beta}(\ss)\to L_2(\ss))=e^{-\beta(m+1)^{\az}},\end{equation*}we get
$$n(\vz_d,d)=\inf\{n\in\Bbb N\ |\ e(n,d)\le \vz_d\}=C(d,m_0).$$ We have for any $t<t_0=\frac {m_0} {1+\beta (m_0)^\az}$,
\begin{align*}\lim_{d\to \infty}\frac{n(\vz_d,d)}{e^{t(1+\ln
d)(1+\ln\frac1{\vz_d})}}=\lim_{d\to\infty}
\frac{C(d,m_0)}{(ed)^{t(1+\ln \frac 1{\vz_d})}}\ge
\lim_{d\to\infty} \frac {d^{m_0-t(1+\ln\frac1{\vz_d})\frac{\ln
ed}{\ln d}}}{m_0\,!}.
\end{align*}Noting that
$$\lim_{d\to\infty}\Big(m_0-t(1+\ln\frac1{\vz_d})\frac{\ln
ed}{\ln d}\Big)=(1+\beta (m_0)^\az)(t_0-t)>0, $$ we obtain
$$\lim_{d\to \infty}\frac{n(\vz_d,d)}{e^{t(1+\ln
d)(1+\ln\frac1{\vz_d})}}=+\infty,$$ which means that \eqref{t1.4}
is not true with $t<t_0$. Hence, we have for $\az>1$,
$$t^{\rm qpol}=t_0.$$

The proof of Theorem \ref{t5.2} is complete.
\end{proof}

\begin{rem} From \eqref{4.12} we know that
$$e(n,d)\le C_1(d)e^{-C_2(d)n^{\az/d}},$$where $C_1(d),\,C_2(d)$ are two constants depending on $d$ but not on $n$. This means that the approximation problem $I_d:
G^{\az,\beta}(\ss)\rightarrow L_2(\ss)$ has exponential
convergence (see \cite{PW}, \cite{IKPW}). So we can consider
exponential convergence tractability. Exponential convergence
tractability has been considered in many papers (see for example,
\cite{PW}, \cite{IKPW}, \cite{PPW}). For $t,s>0$, we say the
approximation problem is $(t,\ln^s)$-weakly tractable (see
\cite{PPW}),  if
\begin{equation*} \lim_{\varepsilon ^{-1}+d\rightarrow \infty
}\frac{\ln n(\varepsilon ,d)}{(\ln\varepsilon ^{-1})^s+d^t}=0.
\end{equation*}Otherwise, the approximation problem is called $(t,\ln^{s})$-intractable. Specially if $s=t=1$, $(t,\ln^{s})$-weak tractability is just
 exponential convergence-weak tractability.
Similarly we can define exponential convergence-uniformly weak
tractability.

Let  $n^{G^{\az,\beta}}(\vz,d)$ and  $n^{H^{\az,\#}}(\vz,d)$ be
the information-complexities of the approximation problems $I_d:
G^{\az,\beta}(\ss)\rightarrow L_2(\ss)$ and $I_d:
H^{\az,\#}(\ss)\rightarrow L_2(\ss)$.

For $0<\vz<1$, we choose $m\in\Bbb N$ such that
$$e^{-\beta(m+1)^{\az}}\le \vz<e^{-\beta m^{\az}}.$$Then
$$n^{G^{\az,\beta}}(\vz,d)=C(d,m).$$
It follows that
$$\lim_{\frac 1\vz+d\to\infty}\frac{\ln(n^{G^{\az,\beta}}(\vz,d))}{(\ln
\vz^{-1})^s+d^t}= \lim_{m+d\to\infty}\frac{\ln(C(d,m))}{\beta^s
m^{s\az}+d^t}.$$ Similarly, for $0<\vz<1$, let $k\in\Bbb N$ be
such that $$(k+2)^{-\az}\le \vz< (k+1)^{-\az}.  $$Then we have
$$n^{H^{\az,\#}}(\vz,d)=C(d,k),$$and
$$\lim_{\frac 1\vz+d\to\infty}\frac{\ln(n^{H^{\az,\#}}(\vz,d))}{
\vz^{-s}+d^t}= \lim_{k+d\to\infty}\frac{\ln(C(d,k))}{
(k+1)^{s\az}+d^t}.$$ Clearly, $$\lim_{\frac
1\vz+d\to\infty}\frac{\ln(n^{G^{\az,\beta}}(\vz,d))}{(\ln
\vz^{-1})^s+d^t}=0\ \ {\rm if \ and \ only\ if}\ \ \lim_{\frac
1\vz+d\to\infty}\frac{\ln(n^{H^{\az,\#}}(\vz,d))}{
\vz^{-s}+d^t}=0,$$which means that the approximation problem $I_d:
G^{\az,\beta}(\ss)\rightarrow L_2(\ss)$  is $(t,\ln^s)$-weakly
tractable if and only if the approximation problem $I_d:
H^{\az,\#}(\ss)\rightarrow L_2(\ss)$ is $(s,t)$-weakly tractable,
and if and only if   $\alpha>1/s$ and $t>0$ or $s>0$ and $t>1$.
  Hence, the approximation problem
$I_d: G^{\az,\beta}(\ss)\rightarrow L_2(\ss)$ is
$(t,\ln^s)$-weakly tractable if and only if   $\alpha>1/s$ and
$t>0$ or $s>0$ and $t>1$. Specially, it  is exponential
convergence-weakly tractable if and only if $\az>1$.

Using the same reasoning, we can prove  that the approximation
problem $I_d: G^{\az,\beta}(\ss)\rightarrow L_2(\ss)$ is
exponential convergence-uniformly weakly tractable if and only if
the approximation problem $I_d: H^{\az,\#}(\ss)\rightarrow
L_2(\ss)$ is uniformly weakly tractable. However, the
approximation problem $I_d: H^{\az,\#}(\ss)\rightarrow L_2(\ss)$
is not uniformly weakly tractable. Hence, the approximation
problem $I_d: G^{\az,\beta}(\ss)\rightarrow L_2(\ss)$ is not
exponential convergence-uniformly weakly tractable for any $\az,
\beta>0$.
\end{rem}

\section{Asymtotics, Preasymtotics, and tractability  on the ball}

\subsection{Sobolev spaces and Gevrey type spaces on the ball}

\

 Let $\bb=\{x\in\mathbb{R}^d:\    \  |x|\le 1\}$ denote the unit
ball in $\mathbb{R}^d$,
 where $x\cdot y$ is
the usual inner product, and $|x|=(x\cdot x)^{1/2}$ is the usual
Euclidean norm.   For the weight $W_\mu(x)=(1-|x|^2)^{\mu-1/2}\
(\mu\ge 0)$, denote by $L_{2,\mu}(\bb) \equiv L_2(\bb,
W_\mu(x)\,dx)$ the space of measurable functions defined on $\bb$
with the finite norm $$ \|f|L_{2,\mu}(\bb)
\|:=\Big(\int_{\bb}|f(x)|^2\,W_\mu(x)dx\Big)^{1/2}.$$ When
$\mu=1/2$, $W_\mu(x)\equiv 1$, and $L_{2,1/2}(\bb)$ recedes to
$L_2(\bb)$.

We denote by $\Pi_n^d(\bb)$ the space of all polynomials in $d$
variables of degree at most $n$ restricted to $\bb$, and by
$\mathcal{V}_{n,\mu}^d(\bb)$ the space of all polynomials of
degree $n$ which are orthogonal to polynomials of low degree in
$L_{2,\mu}(\bb)$. Note that
\begin{equation}\label{6.1} N(n,d):={\rm dim}\,\mathcal{V}_{n,\mu}^d(\bb)=\binom {n+d-1} n.\end{equation} and
\begin{equation}\label{6.2} D(n,d):={\rm dim}\,\Pi_n^d(\bb)=\binom {n+d} n.\end{equation}
It is well known (see \cite[p. 38 or p. 229]{DX} or \cite[p.
268]{DaiX}) that the spaces $\mathcal{V}_{n,\mu}^d(\bb)$ are just
the eigenspaces corresponding to the eigenvalues $-n(n+2\mu+d-1)$
of the second-order differential operator
$$D_{\mu,d}:=\triangle-(x\cdot\nabla)^2-(2\mu+d-1)\,x\cdot
\nabla,$$ where the $\triangle$ and $\nabla$ are the Laplace
operator and gradient operator respectively. More precisely,
$$D_{\mu,d} P=-n(n+2\mu+d-1)P\ {\rm for} \ P\in \mathcal{V}_{n,\mu}^d(\bb).$$
Also, the spaces $\mathcal{V}_{n,\mu}^d(\bb)$  are mutually
orthogonal in $L_{2,\mu}(\bb)$ and
\begin{equation*}
 L_{2,\mu}(\bb)= \bigoplus_{n=0}^\infty
\mathcal{V}_{n,\mu}^d(\bb), \quad \ \ \ \  \Pi_n^d(\bb)=
\bigoplus_{k=0}^n \mathcal{V}_{n,\mu}^{d}(\bb) .
\end{equation*} Let
$$\{\phi_{nk}\equiv \phi_{nk}^{d,\mu}\ |\ k=1,\dots, N(n,d)\}$$ be a
fixed orthonormal basis for $\mathcal{V}_{n,\mu}^d(\bb)$ in
$L_{2,\mu}(\bb)$. Then
$$\{\phi_{nk}\ |\ k=1,\dots, N(n,d),\ n=0,1,2,\dots\}$$ is
an orthonormal basis for $L_{2,\mu}(\bb)$. Evidently,  any $ f\in
L_{2,\mu}(\bb)$ can be expressed by its Fourier series
\begin{equation*}f=\sum_{n=0}^\infty Proj_{n,\mu} f=\sum_{n=0}^\infty{\sum_{k=1}^{N(n,d)}} \lb
\phi_{nk},f\rb_\mu\phi_{nk}, \end{equation*}
 where $Proj_{n,\mu}(f)={\sum\limits_{k=1}^{N(n,d)}} \lb
\phi_{nk},f\rb_\mu\phi_{nk}$ is the orthogonal  projection of $f$
from $L_{2,\mu}(\bb)$ onto $\mathcal{V}_{n,\mu}^d(\bb)$, and
$$ \lb
\phi_{nk},f\rb_\mu :=\int_{\bb} f(x) \phi_{nk}W_\mu(x)\, dx$$ are
the  Fourier coefficients of $f$. We have the following Parseval
equality:
$$\|f\,|\,L_{2,\mu}(\bb)\|=\Big(\sum_{n=0}^\infty{\sum_{k=1}^{N(n,d)}} |\lb
\phi_{nk},f\rb_\mu|^2\Big)^{1/2}.$$

\begin{defn}\label{d6.1}Let $\Lz=\{\lz_k\}_{k=0}^\infty$ be  a non-increasing
positive sequence with $\lim\limits_{k\to\infty}\lz_k=0$, and let
$T^\Lz$ be a multiplier operator on $L_{2,\mu}(\bb)$ defined by
$$T^\Lz (f)=\sum_{n=0}^\infty{\sum_{k=1}^{N(n,d)}} \lz_n\lb
\phi_{nk},f\rb_\mu\phi_{nk}.$$
 We define the
multiplier space $H_{\mu}^\Lz(\bb)$ by
\begin{align*}H_{\mu}^\Lz(\bb)&=\Big\{T^\Lz f\,    \big|\, f\in L_{2,\mu}(\bb) \
{\rm and }\ \|T^\Lz
f\,\big|\,H_{\mu}^\Lz(\bb)\|=\|f|L_{2,\mu}(\bb)\|<\infty\Big\}\\&:=\Big\{f\in
L_{2,\mu}(\bb)\,    \big|\,   \|f\,\big|\,H_{\mu}^\Lz(\bb)\|:=
\Big(\sum_{n=0}^{\infty}\frac1{\lz_n^2}\sum_{k=1}^{N(n,d)} |\lb
\phi_{nk},f\rb_\mu|^2\Big)^{1/2}<\infty\Big\}.
\end{align*}\end{defn}

 Similar to the case on the sphere, we can define the Sobolev
 spaces $H^{r,\square}_\mu(\bb),\ r>0,\ \square\in\{*,+,\#,-\}$
 and the Gevrey type spaces $G_\mu^{\alpha,\beta}(\bb),\
 \az,\beta>0$ analogously. However, in this section we deal only
 with the  most important and interesting case $\mu=1/2$,  and the corresponding spaces $H^{r,*}(\bb)$ and $G^{\alpha,\beta}(\bb)$.

We remark that  there is no difference for the cases $\mu=1/2$ and
$\mu\neq 1/2$ concerning with results about strong equivalences,
asymptotics and preasymptotics, and tractability. We also remark
that the corresponding results on $H^{r,\square}_\mu(\bb),\
\square\in\{+,\#,-\}$ is similar to the ones on
$H^{r,\square}(\ss),\  \square\in\{+,\#,-\}$.  The proofs go
through with hardly any change.

\begin{defn}\label{d6.2}Let $r>0$. The  Sobolev space $H^{r,*}(\bb) \equiv H_{1/2}^{r,*}(\bb)$ is the collection of all $f\in L_2(\bb)$ such
that
$$\|f\,\big|\,H^{r,*}(\bb)\|:=
\Big(\sum_{n=0}^{\infty}\frac1{1+(n(n+d))^r}\sum_{k=1}^{N(n,d)}
|\lb \phi_{nk},f\rb_{1/2}|^2\Big)^{1/2}<\infty,$$ If we set
$\tilde \Lz^*=\{\tilde r_{k,d}^*\}_{k=0}^\infty,\ \tilde
r_{k,d}^*=(1+(k(k+d))^r)^{-1/2}$, then the Sobolev space
$H^{r,*}(\bb)$ is just the multiplier space
$H_{1/2}^{\tilde\Lz^*}(\bb)$.
\end{defn}

\begin{rem} Given $r>0$, we define the fractional power $(-D_{1/2,d})^{r/2}$ of
the operator $-D_{1/2,d}$
 on $f$  by
$$(-D_{1/2,d})^{r/2} (f)= \sum_{k=1}^\infty (k(k+d))^{r/2} Proj_{k,1/2}(f) $$
in the sense of distribution. Then for $f\in H^{r,*}(\bb)$, we
have
$$\|f\,\big|\,H^{r,*}(\bb)\|= \Big(\|f\,|\,L_{2}(\bb)\|^2+\|(-D_{1/2,d})^{r/2}
(f)\,|\,L_{2}(\bb)\|^2\Big)^{1/2}.$$
\end{rem}

\begin{defn}\label{d6.3}Let $\az,\beta>0$. The Gevrey type space $G^{\az,\beta}(\bb)\equiv G_{1/2}^{\az,\beta}(\bb) $ is the collection of all $f\in L_2(\bb)$ such
that
$$\|f\,\big|\,G^{\az,\beta}(\bb)\|:=
\Big(\sum_{n=0}^{\infty}e^{2\beta n^\az}\sum_{k=1}^{N(n,d)} |\lb
\phi_{nk},f\rb_{1/2}|^2\Big)^{1/2}<\infty.$$ If we set
$\Lz_{\az,\beta}=\{e^{-\beta k^\az}\}_{k=0}^\infty$, then the
Gevrey type space $G^{\az,\beta}(\bb)$ is just the multiplier
space $H_{1/2}^{\Lz_{\az,\beta}}(\bb)$. \end{defn}

\subsection{Strong equivalences  on the
ball}

\

Let $\Lz=\{\lz_k\}_{k=0}^\infty$ be  a non-increasing positive
sequence with $\lim\limits_{k\to\infty}\lz_k=0$.  We note from
Definition \ref{d6.1} that the mutiplier space $H_{\mu}^\Lz(\bb)\
(\mu\ge 0)$ is also of form $H^\tau=(L_{2,\mu}(\bb))^\tau$ with
\begin{align*}
\{\tau_k\}_{k=0}^\infty=\{\lz_{0,d},\underbrace{\lz_{1,d},\cdots,\lz_{1,d}}_{N(1,d)},\underbrace{\lz_{2,d},
\cdots,\lz_{2,d}}_{N(2,d)},\cdots,\underbrace{\lz_{k,d},\cdots,\lz_{k,d}}_{N(k,d)},\cdots\},
\end{align*} where $N(m,d)$ and $D(m,d)$ are given in \eqref{6.1}
and \eqref{6.2}. According to Lemma \ref{l2.1} we obtain

\begin{thm} \label{t6.1}For $D(k-1,d)<n\le D(k,d),\ k=0,1,2,\dots,$ we have
$$a_n(T^\Lz: L_{2,\mu}(\bb)\to L_{2,\mu}(\bb))=a_n(I_d: H_{\mu}^\Lz(\bb)\to L_{2,\mu}(\bb))=\lz_{k,d},\ \ \mu\ge 0.$$
where we set $D(-1,d)=0$.

Specially, let $r>0$ and $\az,\beta>0$. Then for $D(k-1,d)<n\le
D(k,d),\ k=0,1,2,\dots,$ we have
\begin{equation}\label{6.3}a_n(I_d: H^{r,*}(\bb)\to L_2(\bb))=(1+(k(k+d))^r)^{-1/2},\end{equation}and
 \begin{equation}\label{6.4}a_n(I_d: G^{\az,\beta}(\bb)\to L_2(\bb))=e^{-\beta k^\az}.\end{equation} \end{thm}

 \begin{thm}\label{t6.2}Suppose that $\lim\limits_{k\to\infty}\lz_{k,d}k^s=1$ for some $s>0$. Then
\begin{equation}\label{6.5} \lim_{n\to\infty}n^{s/d} a_n(I_d: H_{\mu}^\Lz(\bb)\to
L_{2,\mu}(\bb))=\Big(\frac1{d\,!}\Big)^{s/d},\ \mu\ge
0.\end{equation}Specially, we have for $r>0$,
\begin{equation}\label{6.6}\lim_{n\to\infty}n^{r/d}a_n(I_d: H^{r,*}(\bb)\to
L_2(\bb))=\Big(\frac1{d\,!}\Big)^{r/d}\end{equation}
\end{thm}

\begin{proof} The   proof  is similar to the one of \eqref{3.1}.
 For $D(k-1,d)<n\le D(k,d),\ k=0, 1,2,\dots,$ we have
$$a_n(I_d: H_{\mu}^\Lz(\bb)\to L_{2,\mu}(\bb))=\lz_{k,d},\ \mu\ge 0,$$where
$D(k,d)=\binom{k+d}k$. It follows that
$$(D(k-1,d))^{s/d}\lz_{k,d}\le n^{s/d} a_n(I_d: W_{2,\mu}^\Lz(\bb)\to L_{2,\mu}(\bb))\le (D(k,d))^{s/d}\lz_{k,d}. $$
Using the argument of \eqref{3.1} and noting that
$\lim\limits_{k\to\infty}D(k,d)k^{-d}=\frac 1{d\,!} $, we get
\eqref{6.5}. Theorem \ref{t6.2} is proved.
 \end{proof}

\begin{rem}   One can rephrase \eqref{6.6}  as a strong
equivalences $$a_n(I_d: H^{r,*}(\bb)\to L_2(\bb))\sim
n^{-r/d}\Big(\frac1{d\,!}\Big)^{r/d}
$$for $r>0$. The novelty of Theorems
\ref{t6.2}  is that they give a strong equivalence of $a_n(I_d:
H^{r,*}(\bb)\to L_2(\bb)) $ and provide asymptotically optimal
constants, for arbitrary fixed $d$ and $r>0$.
\end{rem}

\begin{thm}\label{t6.3}Let $0<\az<1$, $\beta>0$, and $\tilde \gamma=\Big(\frac1{d\,!}\Big)^{-\az/d}$. Then we have
\begin{equation}\label{6.7}\lim_{n\to \infty }e^{\beta \tilde\gz n^{\az/d}}a_n(I_d: G^{\az, \beta}(\bb)\to L_2(\bb))=1.\end{equation}\end{thm}
\begin{proof}It follows from \eqref{6.4} that for $D(k-1,d)<n\le D(k,d),\
k=0,1,2,\dots,$
\begin{equation*} a_n\equiv a_n(I_d: G^{\az,\beta}(\bb)\to
L_2(\bb))=e^{-\beta k^\az}.\end{equation*}Therefore, we have
\begin{equation*} e^{\beta \tilde\gz (D(k-1,d))^{\az/d}}e^{-\beta k^\az}<  e^{\beta \tilde\gz n^{\az/d}}
a_n\le e^{\beta \tilde\gz (D(k,d))^{\az/d}} e^{-\beta
k^\az}.\end{equation*} Since for $0<\az<1$,
$$\lim_{k\to\infty}(\tilde \gz (D(k,d))^{\az/d} -k^\az)=\lim_{k\to\infty}k^\az\Big(\Big( \prod_{j=1}^{d}(1+\frac jk)\Big)^{\frac \az d}-1\Big)=0,$$
similar to the proof of \eqref{3.3}, we get \eqref{6.7}. Theorem
\ref{t6.3} is proved.
\end{proof}

\begin{rem}   One can rephrase \eqref{6.7} as  a strong
equivalence
$$a_n(I_d: G^{\az, \beta}(\bb)\to L_2(\bb))\sim e^{-\beta \tilde\gz
n^{\az/d}}$$ for $0<\az<1$ and $\beta>0$, where
$\tilde\gamma=\Big(\frac1{d\,!}\Big)^{-\az/d}$. The novelty of
Theorems \ref{t6.3} is that they give a strong equivalence of $
a_n(I_d: G^{\az, \beta}(\bb)\to L_2(\bb))$ and provide
asymptotically optimal constants, for arbitrary fixed $d$,
$0<\az<1$, and $\beta>0$.
\end{rem}

\subsection{Preasymptotics and asymptotics  on the ball}

\

Let $r>0$ and $\az,\beta>0$. Then for $D(m-1,d)<n\le D(m,d),\
m=0,1,2,\dots,$ we have
\begin{equation*}a_n(I_d: H^{r,*}(\bb)\to L_2(\bb))=(1+(m(m+d))^r)^{-1/2},\end{equation*}and
 \begin{equation*}a_n(I_d: G^{\az,\beta}(\bb)\to L_2(\bb))=e^{-\beta m^\az},\end{equation*}
where $D(-1,d)=0$ and $D(m,d)=\binom {d+m}d$. We note that
\begin{align*}\max\Big\{\big(1+\frac{m}{d}\big)^{d},\,\big(1+\frac{d}{m}\big)^{m}\Big\}\leq
D(m,d) \le
\min\Big\{e^d\big(1+\frac{m}{d}\big)^{d},\,e^m\big(1+\frac{d}{m}\big)^{m}\Big\}.
\end{align*}
Using the same reasoning as in the proof of Theorem \ref{t4.1}, we
obtain that for $D(m-1,d)<n\le D(m,d),\ 1\le m \le d,$
$$  m\asymp \frac{ \log n}{1+\log
\big(\frac {d}{\log n}\big)},$$and for $D(m-1,d)<n\le D(m,d),\ m>
d,$ $$m\asymp dn^{1/d}.$$By the above two equivalences  we can
obtain the following two theorems.

\begin{thm}\label{t6.4} Let $r>0$.
We have
\begin{equation}\label{6.8}
a_n(I_d: H^{r,*}(\bb)\to L_2(\bb)) \asymp
\left\{\begin{matrix} 1,\ \ \ &n=1,\\
d^{-r/2}, & \ \  2\le n\leq d,\\
 d^{-r/2}\Big(\frac{\log(1+\frac{d}{\log n})}{\log n}\Big)^{r/2},&\ \  d\le n\le 2^d, \\
 d^{-r}n^{-r/d},&\ \  n\ge 2^d,
\end{matrix}\right.
\end{equation}
where the equivalence constants depend only on $r$, but not on $d$
and $n$.
\end{thm}

\begin{thm}\label{t6.5} Let $\alpha,\beta>0$.
We have
\begin{equation}\label{6.9}
\ln \big(a_n(I_d: G^{\az,\beta}(\bb)\to L_2(\bb))\big) \asymp
-\beta
 \left\{\begin{matrix}
1, & \ \  1\le n\leq d,\\
 \Big(\frac{\log n}{\log(1+\frac{d}{\log n})}\Big)^{\az},&\ \  d\le n\le 2^d, \\
 d^{\az}n^{\az/d},&\ \  n\ge 2^d,
\end{matrix}\right.
\end{equation}
 where the equivalence constants depend only on $\az$, but not on
$d$ and $n$.
\end{thm}

\subsection{Tractability on the ball}

\

Using the same methods as in Theorems \ref{t5.1} and \ref{t5.2},
we obtain the two theorems.
\begin{thm}\label{t6.6}
Let $r>0$ and  $s,\,t>0$. Then
 the approximation problem
$$I_d: H^{r,*}(\bb)\rightarrow L_2(\bb)$$
is $(s, t)$-weakly tractable if and only if $s>1/r$ and $t>0$ or
$s>0$ and $t>1$. Specially,  the approximation problem $I_d:
H^{r,*}(\bb)\rightarrow L_2(\bb)$
 is weakly tractable if and only if $r>1$, not uniformly
weakly tractable,  and does not suffer from the curse of
dimensionality.
\end{thm}

\begin{thm}\label{t6.7}
Let $\az,\beta>0$. Then
 the approximation problem
$$I_d: G^{\az,\beta}(\bb)\rightarrow L_2(\bb)$$

(1)  is uniformly   weakly tractable.

(2)   is not polynomially tractable.

 (3)   is quasi-polynomially tractable if and only if $\az\ge 1$ and  the exponent
of quasi-polynomial tractability is $$t^{\rm qpol}=\sup_{m\in \Bbb
N}\frac m {1+\beta m^\az},\ \ \az\ge1.$$Specially, if $\az=1$,
then $t^{\rm qpol}=\frac1\beta$.\end{thm}

\begin{rem} We can also consider
exponential convergence tractability for the approximation problem
$$I_d: G^{\az,\beta}(\bb)\rightarrow L_2(\bb)\ \ (\az,\beta>0).$$

We can prove that  the approximation problem $I_d:
G^{\az,\beta}(\bb)\rightarrow L_2(\bb)$  is $(t,\ln^s)$-weakly
tractable if and only if   $\alpha>1/s$ and $t>0$ or $s>0$ and
$t>1$, and  is not exponential convergence-uniformly weakly
tractable for any $\az, \beta>0$. Specially, it  is exponential
convergence-weakly tractable if and only if $\az>1$.
\end{rem}

\section*{Acknowledgments}
 The  authors were supported
  by the National Natural Science
Foundation of China (Project no.  11671271, 11271263),
 the  Beijing Natural Science Foundation (1172004, 1132001).

\end{document}